\documentclass{article}

\usepackage{amsmath, amsthm}%
\usepackage{amsfonts}%
\usepackage{amssymb}%
\usepackage{graphicx}
\usepackage{amscd, amstext, hyperref}
\usepackage{pictexwd}
\usepackage{dcpic}
\usepackage[mathscr]{eucal}

\theoremstyle{definition}
\newtheorem{theorem}{Theorem}[section]
\newtheorem{proposition}{Proposition}[section]
\newtheorem{remark}{Remark}[section]
\newtheorem{example}{Example}[section]
\newtheorem{definition}{Definition}[section]

\newtheorem{corollary}{Corollary}[section]
\newtheorem{lemma}{Lemma}[section]
\newtheorem{question}{Question}[section]

\begin{document}

\title{Symplectic Automorphisms and the Picard Group of a K3 Surface}
\author{Ursula Whitcher \thanks{I thank the referee for comments which improved the exposition of the paper, Alice Garbagnati, Paul Hacking, and Kenji Hashimoto for useful discussion, Charles Doran for his generous guidance, and the organizers of the 2007 GAeL conference in Istanbul, where I presented an early version of this work, for their support.  Further, I gratefully acknowledge partial support in the preparation of this work by the National Science Foundation under the Grant DMS-083996.}}
\date{}
\maketitle

\section{Introduction}

Let $X$ be a K3 surface, and let $G$ be a finite group acting on $X$ by automorphisms.  
The action of $G$ on $X$ induces an action on the cohomology of $X$.  We assume $G$ acts symplectically: that is, $G$ acts as the identity on $H^{2,0}(X)$.  In this case, the minimum resolution $Y$ of the quotient $X/G$ is itself a K3 surface.

In \cite{Nikulin}, Nikulin classified the finite abelian groups which act symplectically on K3 surfaces by analyzing the relationship between $X$ and $Y$.  Nikulin also described moduli spaces of K3 surfaces with $G$ actions for the case that $G$ is an abelian group; these topological spaces are subspaces of the moduli space of marked K3 surfaces.  Mukai showed in \cite{Mukai} that any finite group $G$ with a symplectic action on a K3 surface is a subgroup of a member of a list of eleven groups, and gave an example of a symplectic action of each of these maximal groups.  Xiao gave an alternate proof of the classification in \cite{Xiao} by listing the possible types of singularities, and Kond{\=o} showed in \cite{Kondo} that the action of $G$ on the K3 lattice extends to an action on a Niemeier lattice.  

The Picard group of $X$ has a primitive sublattice $S_G$ determined by the action of $G$.  The rank of $S_G$ varies from $8$ to $19$, depending on $G$.  Thus, K3 surfaces which admit symplectic group actions provide a rich source of examples of families of K3 surfaces with high-rank Picard groups.  The monodromy and mirror symmetry properties of algebraic K3 surfaces which admit a sublattice $S_G$ of rank $18$, and therefore have a Picard group of rank $19$,  have been extensively studied.  (cf. \cite{NarShiga, Smith, DorKerr})  Conversely, if the structure of $\mathrm{Pic}(X)$ is known, one may examine its sublattices to detect symplectic group actions on $X$.  Morrison used the structure of $S_G$ for $G = \mathbf{Z}/2\mathbf{Z}$ to study K3 surfaces which admit Shioda-Inose structures in \cite{Morrison}.  Recently, Garbagnati and Sarti have computed $S_G$ for all possible abelian groups with symplectic action, correcting an earlier computation of Nikulin's; Garbagnati has also studied $S_G$ for dihedral groups, and Hashimoto calculated the invariants of $S_G$ for the permutation group $G=\mathcal{S}_5$ (see \cite{GarSarti, Garbagnati, Gar08b, Gar09, Hashimoto}).

In Section~\ref{S:sublattice}, we discuss the relationship between the lattice $S_G$ and the singularities of $X/G$ for any symplectic $G$-action, and show how to compute the rank and discriminant of $S_G$.  In Section~\ref{S:moduli}, we show that the maps between $X$, $Y$, and $X/G$ can be generalized to the realm of moduli spaces, and describe moduli spaces of K3 surfaces with symplectic $G$-action.  The key observation is that we may work backwards from a K3 surface $Y$ endowed with a set of exceptional curves to the K3 surface $X$.  We use our moduli spaces to develop techniques for classifying all possible symplectic actions of a group $G$.

\section{A Sublattice of the Picard Group}\label{S:sublattice}

The cup product induces a bilinear form $\langle \, , \rangle$ on $H^2(X,\mathbf{Z}) \cong H \oplus H \oplus H \oplus E_8 \oplus E_8$.  (We take $E_8$ to be negative definite.)  Using this form, we define $S_G = (H^2(X,\mathbf{Z})^G)^\perp$.  The Picard group of $X$, $\mathrm{Pic}(X)$,  consists of $H^{1,1}(X) \cap H^2(X,\mathbf{Z})$; the group $\mathcal{T}(X) \subseteq H^2(X,\mathbf{Z})$ of transcendental cycles is defined as $(\mathrm{Pic}(X))^\perp$.  Nikulin showed that the groups $\mathrm{Pic}(X)$ and $S_G$ are related:

\begin{proposition}\label{P:Nik} \cite[Lemma 4.2]{Nikulin}
$S_G \subseteq \mathrm{Pic}(X)$ and $\mathcal{T}(X) \subseteq H^2(X,\mathbf{Z})^G$.  The lattice $S_G$ is nondegenerate and negative definite.
\end{proposition}

\noindent In this section, we show how to compute the rank and discriminant of $S_G$, and relate $S_G$ to the singularity structure of $X/G$.

The number of fixed points of an element $g$ of a group $G$ acting symplectically on a K3 surface $X$ depends only on the order of $g$, by the results of \cite[\S 5]{Nikulin}.

\begin{proposition}\label{P:MuOg} \cite[\S 3]{Mukai} \cite[Proposition 4.5]{Oguiso} Let $m(n)$ be the number of elements in $G$ of order $n$, and let $f(n)$ be the number of fixed points of an element of order $n$.  Then,
\[\mathrm{rank} \; H^*(X,\mathbf{Z})^G = \frac{1}{|G|} (24 + \sum_{n=2}^8 m(n)f(n) ).\]
\end{proposition}

\noindent Since $G$ acts as the identity on $H^0(X, \mathbf{Z})$ and $H^4(X,\mathbf{Z})$ as well as $H^{2,0}(X)$ and $H^{0,2}(X)$, we also know that $\mathrm{rank} \; H^*(X,\mathbf{Z})^G \geq 4$.

Because $G$ acts symplectically on $X$, $X/G$ has a minimal resolution $Y$ which is also a K3 surface.  Let $\{p_i\}$ be the singular points of $X/G$.  The inverse image in $Y$ of $p_i$ is a configuration $\Psi_i$ of $(-2)$-curves of type $A_l$, $D_m$ or $E_n$; let $c_i$ be the number of curves in this configuration.  The configurations $\Psi_i$ generate a lattice $K$ in $\mathrm{Pic}(Y) \subset H^2(Y,\mathbf{Z})$ of rank $\sum_i c_i$.  Let $M$ be the minimal primitive sublattice of $H^2(Y,\mathbf{Z})$ containing $K$.  Then $M$ also has rank $\sum_i c_i$, and $H^2(Y,\mathbf{Z})/M$ is a free abelian group.  Xiao showed in \cite[Lemma 6]{Xiao} that $M$ is uniquely determined by the $\Psi_i$.   

\begin{remark} \cite[Theorem 3]{Xiao}
If $G$ is isomorphic to $Q_8$, the group of unit quaternions, or $T_{24}$, the binary tetrahedral group of order $24$, then $K$ may be one of two different lattices, depending on the action of $G$.  In all other cases, $K$ (and thus $M$) is uniquely determined by $G$. 
\end{remark}

Let $\{q_{ij}\}$ be the inverse images in $X$ of $p_i$, and let $G_i$ be the stabilizer group of any $q_{ij}$; set $N_i = |G_i|$.

\begin{proposition}\label{P:Xiao} \cite[Lemma 1]{Xiao}
\[\sum_i c_i = \frac{24(|G|-1)}{|G|} - \sum_{i=1}^{k} \frac{N_i-1}{N_i}.\]
\end{proposition}

\begin{proposition}\label{P:NikWhit}
$\mathrm{rank} \; S_{G} = \sum_i c_i.$
\end{proposition}

Nikulin discusses this proposition for the case that $G$ is abelian in \cite[\S 10]{Nikulin}.  We use Propositions \ref{P:MuOg} and \ref{P:Xiao} to give a brief proof for any $G$.

\begin{proof}
We calculate:
\begin{align}
\mathrm{rank} \; H^*(X,\mathbf{Z})^G + \sum_i c_i & = 24 - \mathrm{rank} \; S_{G} + \sum_i c_i \notag \\
& = \frac{1}{|G|} (24 + \sum_{n=2}^8 m(n)f(n) ) + \frac{24(|G|-1)}{|G|} - \sum_{i=1}^{k} \frac{N_i-1}{N_i} \notag \\
& = 24 + \frac{1}{|G|} \sum_{n=2}^8 m(n)f(n) - \sum_{i=1}^{k} \frac{N_i-1}{N_i}. \notag
\end{align}

Thus, it suffices to show that 
\[\sum_{n=2}^8 m(n)f(n) = \sum_{i=1}^{k} \frac{|G|}{N_i}(N_i-1).\]

$\sum_{n=2}^8 m(n)f(n)$ counts each non-identity element $g$ of $G$ once for each point of $X$ which $g$ fixes.  $N_i - 1$ counts the non-identity elements of the stabilizer group $G_i$.  The point $p_i$ has precisely $\frac{|G|}{N_i}$ preimages $q_{ij}$ in $X$; by definition, the elements of $G_i$ fix the $q_{ij}$.  Summing over all singular points $p_i$, we see that $\sum_{i=1}^{k} \frac{|G|}{N_i}(N_i-1)$ also counts every element of $G$ other than the identity once for each point of $X$ which that element fixes.
\end{proof}

Though the lattices $S_G$ and $M$ are primitive sublattices of the K3 lattice $H \oplus H \oplus H \oplus E_8 \oplus E_8$ and have the same rank, they are not isomorphic: by \cite[Lemma 4.2]{Nikulin}, $S_G$ contains no elements with square $-2$.  Instead, the relationship between $S_G$ and $M$ is given by the fact that $S_G = (H^2(X,\mathbf{Z})^G)^\perp$ and the following exact sequence.

\begin{theorem}\label{T:exactSequence}
There exists an exact sequence 
\[\begin{CD}
0 @>>>  M/K @>>> H^2(Y,\mathbf{Z})/K @> \theta >> (H^2(X,\mathbf{Z}))^G @>>> H^3(G,\mathbf{Z}) @>>> 0 \end{CD}\]
where $\langle \theta (m),\theta (n) \rangle = |G| \, \langle m,n \rangle $.
\end{theorem}  

\begin{proof}
Let $X' = X - (\cup_{i,j} q_{ij})$ and let $Y' = Y - (\cup_i \Psi_i)$.  Since $X$ is a simply connected complex surface, $X'$ is also simply connected; since $Y' = X'/G$, $X'$ is the universal covering space of $Y'$.  By \cite[Application XVI.1]{EilCart}, there exists an exact sequence

\[\begin{CD}
0 @>>>  H^2(G,\mathbf{Z})  @>>> H^2(Y',\mathbf{Z}) @> \theta >> (H^2(X',\mathbf{Z}))^G @>>> H^3(G,\mathbf{Z}) @> \zeta >> H^3(Y',\mathbf{Z}) \end{CD}.\]

Since $\theta$ is induced by the quotient map $X' \rightarrow Y'$, $\langle \theta (m),\theta (n) \rangle = |G| \, \langle m,n \rangle $.  Xiao showed in\cite[Lemma 2]{Xiao} that $H^2(G,\mathbf{Z}) = M/K$ and $H^2(Y',\mathbf{Z}) = H^2(Y,\mathbf{Z})/K$; because $X$ is a complex surface (and therefore has four real dimensions), $H^2(X,\mathbf{Z}) = H^2(X',\mathbf{Z})$.  Since $G$ is a finite group, $H^3(G,\mathbf{Z})$ is a finite abelian group.  We shall show that $H^3(Y',\mathbf{Z})$ is a free abelian group, so $\zeta$ must be the zero map.

Let $N_i$ be a tubular neighborhood of the configuration of exceptional curves $\Psi_i$ in $Y$, and let $L_i$ be the boundary of $N_i$.  Consider the Mayer-Vietoris sequence

\[\begin{CD}
\dots  @>>> H^3(Y, \mathbf{Z})  @>>> H^3(Y', \mathbf{Z}) \oplus \bigoplus_i H^3(N_i, \mathbf{Z})  @>>> \bigoplus_i H^3(L_i, \mathbf{Z}) @>>>  H^4(Y, \mathbf{Z})  @>>>  \dots \end{CD}\]
 
Since $Y$ is a K3 surface, $H^3(Y)=0$ and $H^4(Y)=\mathbf{Z}$.  Because $N_i$ is a tubular neighborhood of an ADE configuration of curves, $N_i$ is homotopy equivalent to a bouquet of $c_i$ $2$-spheres, so $H^3(N_i)=0$.  Since $L_i$ is a smooth real $3$-manifold, $H^3(L_i)=\mathbf{Z}$.  Furthermore, the map $\bigoplus_i H^3(L_i) \rightarrow H^4(Y)$ is given by $f: \mathbf{Z}^k \rightarrow \mathbf{Z}$, where $f((x_1, \dots, x_{c_i})) = x_1 +..+x_{k}$.  Thus, $H^3(Y')$ is isomorphic to the kernel of $f$, a free abelian group of rank $k - 1$.

\end{proof}

\begin{remark}
In \cite[Proposition 2.4]{Garbagnati}, Garbagnati proved a variant of Theorem~\ref{T:exactSequence} in the case that $G$ is an abelian group, correcting Nikulin's claim that $\theta$ is surjective. 
\end{remark}


\begin{lemma}\label{L:imageTheta}\cite[Lemma 10.2]{Nikulin}
Let $J = \rm{Im}(\theta)$.  Then the lattice discriminants $d(J)$ and $d(M)$ are related by 
\[d(J) = -\frac{|G|^{22 - \mathrm{rank}(M)}}{d(M)}.\]
\end{lemma}

\begin{example}
Let $X$ be a K3 surface which admits a symplectic action by the permutation group $G = \mathcal{S}_4$.  Then $\mathrm{Pic}(X)$ admits a primitive sublattice $S_G$ which has rank $17$ and discriminant $d(S_G) = -2^6 \cdot 3^2$.
\end{example}

\begin{proof}
According to \cite[Table 2]{Xiao}, when $G = \mathcal{S}_4$, $K$ is the rank $17$ lattice given by $(A_3)^2 \oplus (A_2)^3 \oplus (A_1)^5$, and $M/K \cong \mathbf{Z}/(2 \mathbf{Z})$.   Next we use the fact that if lattices $L$ and $L'$ have the same rank, and $L \subset L'$, then the discriminants $d(L)$ and $d(L')$ are related by $d(L)/d(L') = [L':L]^2$, where $[L':L]$ is the index of $L$ in $L'$ as an abelian group.  Since $d(K) = -2^9 \cdot 3^3$, we see that $d(M) = -2^7\cdot 3^3$.  By Lemma~\ref{L:imageTheta}, the discriminant $d(J) = 2^8\cdot3^2$.  The cohomology group $H^3(\mathcal{S}_4,\mathbf{Z})$ is isomorphic to $\mathbf{Z}/(2 \mathbf{Z})$, so $[(H^2(X,\mathbf{Z}))^G: J] = 2$ and $d((H^2(X,\mathbf{Z}))^G) = 2^6\cdot 3^2$.  Since $S_G$ is the perpendicular complement of $(H^2(X,\mathbf{Z}))^G$ in the unimodular K3 lattice $H \oplus H \oplus H \oplus E_8 \oplus E_8$, we conclude that $d(S_G) = -d((H^2(X,\mathbf{Z}))^G) = - 2^6\cdot 3^2$.
\end{proof}

\begin{example}\label{E:L2of7}
Let $X$ be a K3 surface which admits a symplectic action by the Chevalley group $G = L_2(7) \cong PSL(2,\mathbf{F}_7)$.  Then $(H^2(X,\mathbf{Z}))^G$ has rank $3$ and discriminant $196$.  
\end{example}

\begin{proof}
Consulting \cite[Table 2]{Xiao}, we find that $K$ is the rank $19$ lattice given by $A_6 \oplus (A_3)^2 \oplus (A_2)^3 \oplus A_1$, and $M \cong K$.  Thus, $d(M) = -7 \cdot 4^2 \cdot 3^3 \cdot 2$.  The order of $L_2(7)$ is $2^3 \cdot 3 \cdot 7$, so by Lemma~\ref{L:imageTheta}, the discriminant $d(J) = 2^4 \cdot 7$.  We may use the computer algebra system \cite{SAGE} to show that $H^3(G,\mathbf{Z}) \cong \mathbf{Z}/2\mathbf{Z}$.  Thus, $[(H^2(X,\mathbf{Z}))^G: J] = 2$, so $d(H^2(X,\mathbf{Z})) = (2^4 \cdot 7^2)/2^2 = 196$.
\end{proof}

\begin{remark}
The result of Example~\ref{E:L2of7} is the ``Key Lemma'' of \cite{OZ}; that paper gives a longer and more involved proof by constructing an embedding of $S_G$ in a Niemeier lattice.
\end{remark}

\section{Classifying Symplectic Group Actions}\label{S:moduli}

In \cite{Nikulin}, Nikulin showed that, when $G$ is abelian, symplectic actions of $G$ are unique up to overall isomorphisms.  In this section, we develop techniques for classifying the symplectic actions of any group, and show that certain non-abelian groups admit multiple distinct symplectic actions.  To do so, we construct moduli spaces of K3 surfaces which can be realized as resolutions of quotients by symplectic group actions.  Our discussion extends and refines the constructions of \cite{Nikulin} in the non-abelian case.

We begin by reviewing the standard constructions of moduli spaces of K3 surfaces.  We follow the exposition and notation of \cite[\S VIII]{BHPV}.

We call a choice of isomorphism $\alpha: H^2(X,\mathbf{Z}) \to L$ a \emph{marking} of $X$, and refer to the pair $(X;\alpha)$ as a \emph{marked K3 surface}.  Let us write $\langle\,,\,\rangle$ for the bilinear form on $L$; we set $L_\mathbf{R} = L \otimes \mathbf{R}$ and $L_\mathbf{C} = L \otimes \mathbf{C}$, and extend the bilinear form appropriately.  

For any nonzero element $\omega$ of $L_\mathbf{C}$, let $[\omega]$ be the corresponding element of the projective space $\mathbf{P}(L_\mathbf{C})$.  Let $\Omega = \{[\omega] \in \mathbf{P}(L_\mathbf{C}) \,|\, \langle \omega,\omega \rangle = 0, \langle \omega,\bar{\omega} \rangle > 0\}$.  Let $(X;\alpha)$ be a marked K3 surface, and let $\omega_X$ be a nowhere-vanishing holomorphic two-form on $X$.  (The form $\omega_X$ is unique up to a scalar multiple.)  The image of $\omega_X$ under $\alpha_\mathbf{C}$ determines a point $[\alpha_\mathbf{C}(\omega_X)]$ in $\mathbf{P}(L_\mathbf{C})$.  Since $\omega_X \wedge \omega_X = 0$ and $\omega_X \wedge \bar{\omega}_X > 0$, $[\alpha_\mathbf{C}(\omega_X)]$ is an element of $\Omega$, which we refer to as the \emph{period point}.  There exists a universal marked family of K3 surfaces.  The base space $\mathcal{N}$ is a non-Hausdorff ``smooth analytic space'' of dimension 20.  The period points of marked K3 surfaces yield a \emph{period map} $\tau_\mathcal{N}: \mathcal{N} \to \Omega$.  

We now consider marked K3 surfaces with specified K\"{a}hler class.  We wish to specify the K\"{a}hler class in a manner consistent with our marking.  For any $[\omega] \in \Omega$, let $E(\omega)$ be the oriented 2-plane in $L_\mathbf{R}$ spanned by $\{\mathrm{Re}\,\omega, \mathrm{Im}\,\omega\}$.  Let $K\Omega$ be the set $\{(\kappa,[\omega]) \in L_\mathbf{R} \times \Omega \,|\, \langle \kappa, \lambda \rangle = 0 \; \forall \; \lambda \in E(\omega)\;\mathrm{and}\;\langle \kappa, \kappa \rangle > 0\}$.  Then $K\Omega$ is a fiber bundle over $\Omega$.  For any $[\omega] \in \Omega$, let $C_\omega$ be the cone $\{x \in L_\mathbf{R}\,|\,\langle x,\omega \rangle = 0, \langle x,x \rangle > 0\}$.  We may choose a connected component $C^+_\omega$ of $C_\omega$ in such a way that $C^+_\omega$ varies continuously with our choice of $[\omega]$.  If $(X;\alpha)$ is a marked K3 surface and $\kappa \in H^{1,1}(X)$ a K\"{a}hler class, we say that $(X,\kappa)$ is a \emph{marked pair} if $\alpha_\mathbf{C}(\kappa) \in C^+_\omega$, where $[\omega]$ is the period point of $X$.  There exists a universal object $\mathcal{M}$ for marked pairs and a natural forgetful map $\mathcal{M} \to \mathcal{N}$.  The space $\mathcal{M}$ is a 60-dimensional real-analytic manifold.

Let $(K\Omega)^0$ be the subset of $K\Omega$ consisting of those points $(\kappa, [\omega])$ such that $\langle \kappa, d \rangle \neq 0$ for every $d \in L$ such that $\langle d,d \rangle = -2$ and $\langle \omega,d\rangle = 0$.  The subset $(K\Omega)^0$ is open in $K\Omega$.  We may define a real-analytic map $\tau_\mathcal{M}: \mathcal{M} \to (K\Omega)^0$ called the \emph{refined period map} as follows: if $m \in \mathcal{M}$ corresponds to the marked pair $(X,\kappa)$, we set $\tau_\mathcal{M}(m) = (\alpha_\mathbf{C}(\kappa),[\omega])$.

The period map $\tau_\mathcal{N}$ and the refined period map $\tau_\mathcal{M}$ together with the forgetful maps fit into a commutative diagram:

\[\begin{CD}
\mathcal{M} @>\tau_\mathcal{M}>> (K\Omega)^0 \\
@VVV              @VVV \\
\mathcal{N} @>\tau_\mathcal{N}>> \Omega
\end{CD}\]

\begin{theorem}\label{T:tauInjective}(See \cite[Theorem VIII.12.3]{BHPV} for a modern proof.)
The refined period map $\tau_\mathcal{M}$ is injective.
\end{theorem}

We also have a surjectivity result due to Todorov (see \cite{Todorov}):

\begin{theorem}\label{T:tauSurjective}(cf. \cite[Theorem VIII.14.1]{BHPV})
The refined period map $\tau_\mathcal{M}$ is surjective.
\end{theorem}

\begin{remark}
In \cite{Nikulin}, Nikulin uses a moduli space of K\"{a}hler K3 surfaces which has two components; by working with marked pairs, we have essentially fixed our choice of component. 
\end{remark}


We now describe moduli spaces which will parametrize possible resolutions $Y$.

\begin{definition}\cite[Definition 2.1]{Nikulin}
A \emph{condition} $T$ is a primitive sublattice $M$ of $L$ and a finite subset $\{c_i\}$ of $M$ such that $c_i^2 = -2$ for each $i$.  
\end{definition}

We work with conditions $T$ where $M$ is negative definite.

\begin{definition}\cite[Definition 2.2]{Nikulin}
A \emph{marked K3 surface with condition} $T$ is a K3 surface $Y$ together with an isometry $\alpha: H^2(X,\mathbf{Z}) \to L$ such that $\alpha^{-1}(M) \subset H^{1,1}(Y)$ and $\alpha^{-1}(c_i)$ is represented by a nonsingular rational curve on $Y$ for each $i$.  
\end{definition}

\begin{remark}
A nonsingular rational curve with self-intersection $-2$ in a K3 surface is uniquely determined by its homology class. (cf. \cite[Proposition VIII.3.7]{BHPV})  We will often identify the cohomology classes $\alpha^{-1}(c_i)$ with the corresponding curves.
\end{remark}

\begin{definition}
A \emph{marked pair with condition} $T$ is a marked pair $(Y,\kappa)$ such that $Y$ is a marked K3 surface with condition $T$.
\end{definition}

Any marked pair with condition $T$ must satisfy $\langle \kappa, \alpha^{-1}(c_i) \rangle > 0$ for each $i$, because $\alpha^{-1}(c_i)$ is represented by a nonsingular rational curve.

\begin{definition}
Let $\mathcal{M}_T$ be the subspace of $\mathcal{M}$ corresponding to the marked pairs with condition $T = \{c_j\} \subset M \subset L$.  Let $\mathcal{N}_T$ be the image of $\mathcal{M}_T$ under the forgetful map.
\end{definition}

\begin{remark}
Note that by taking $M$ to be isomorphic to a primitive sublattice of $H^2(Y,\mathbf{Z})$, we have fixed the primitive embedding of $M$ in $L$ up to automorphisms of $L$.  
\end{remark}

Let $(K\Omega)^0_M$ be the subset of $(K\Omega)^0$ given by the refined period points $(\kappa,[\omega])$ such that $M$ is contained in the perpendicular complement of $\omega$ and $\bar{\omega}$.  Suppose $m \in \mathcal{M}$ corresponds to a marked K3 surface $(Y;\alpha)$ with K\"{a}hler class $\kappa$, and suppose $\tau_\mathcal{M}(m) = (\kappa,[\omega])$.  Let $\Delta_m$ be the set given by $\{\delta \in \alpha_\mathbf{C}(H^{1,1}(X)) \,|\, \langle \delta,\delta \rangle = -2\}$, and let $\Delta_m^+$ be the subset of $\Delta_m$ given by $\langle \kappa,\delta \rangle >0$.  Let $(K\Omega)^0_T$ be the subset of $(K\Omega)^0_M$ such that $c_i \in \Delta_m^+$ and $c_i$ is an irreducible element of $\Delta_m^+$ for each $i$.  The following proposition follows immediately.

\begin{proposition}\label{P:MtoKOmega}\cite[Proposition 2.8]{Nikulin}
The point $m \in \mathcal{M}_T$ if and only if $\tau_\mathcal{M}(m) \in (K\Omega)^0_T$.
\end{proposition}

\begin{proposition}\label{P:partitionComponents}\cite[Proposition 2.9]{Nikulin}
Let $M$ be a negative definite lattice with $\mathrm{rank}\;M \leq 19$.  Then $(K\Omega)^0_M$ is a closed smooth complex subspace of $(K\Omega)^0$.  Furthermore, $(K\Omega)^0_M$ is connected, and $(K\Omega)^0_M - (K\Omega)^0_T$ is a closed subset of $(K\Omega)^0_M$ which is the union of at most countably many closed complex subspaces of $(K\Omega)^0_M$.
\end{proposition}

\begin{theorem}\label{T:components}
$\mathcal{M}_T$ is path-connected.
\end{theorem}

\begin{remark}
Nikulin proved a variant of Theorem~\ref{T:components} under the assumption that $\mathrm{rank}~M \leq 18$ by constructing paths between elements (see \cite[Theorem 2.10]{Nikulin}).  We give a quick argument using the bijectivity of the refined period map $\tau_\mathcal{M}$.  
\end{remark}

\begin{proof}
Proposition~\ref{P:partitionComponents} implies that $(K\Omega)^0_T$ is connected and path-connected.  Since the refined period map $\tau_\mathcal{M}$ is injective and surjective, Proposition~\ref{P:MtoKOmega} implies that $\mathcal{M}_T$ is also path-connected.
\end{proof}

\begin{corollary}
$\mathcal{N}_T$ is connected and path-connected.
\end{corollary}

Let $G$ be a group which admits a symplectic action on some marked K3 surface $(X_\nu;\beta_\nu)$.  Let $\dot{G}$ be the group which has the same elements of $G$, but where the group operation is written in the reverse order.  Then $\dot{G}$ acts on $H^2(X, \mathbf{Z})$, and we may use $\beta_0$ to define an embedding $\phi: \dot{G} \hookrightarrow O(L)$.  In the following discussion, we treat the group $G$ and the embedding $\phi$ as fixed.

\begin{definition}\cite[Definition 4.9]{Nikulin}
A \emph{marked K3 surface with symplectic automorphism group $G$ and action $\phi$} on the integral cohomology is a triple $(X,i;\beta)$ such that $(X;\beta)$ is a marked K3 surface, $i: G \hookrightarrow \text{Aut}(X)$ is an embedding where $G$ acts symplectically on $X$, and 
\[\beta \cdot i(g)^* \cdot \beta^{-1} = \phi(g)\]
for any $g \in G$.  We say that two such triples $(X,i;\beta)$ and $(X',i';\beta')$ are \emph{isomorphic} if there exists an isomorphism $t: X \to X'$ such that $\beta' = \beta \cdot t^*$ and $i'(g) = t \cdot i(g) \cdot t^{-1}$ for any $g \in G$.
\end{definition}

\begin{definition}
A \emph{marked pair with symplectic automorphism group $G$ and action $\phi$} is a triple $(X,i;\beta)$ which is a marked K3 surface with symplectic automorphism group $G$ and action $\phi$ together with a K\"{a}hler class $\kappa$ such that $(X, \kappa)$ is a marked pair.  We say that two such pairs $(X, \kappa)$ and $(X', \kappa')$ are \emph{isomorphic} if there exists an isomorphism $t: X \to X'$ such that $t^*(\kappa') = \kappa$ and the underlying triples $(X,i;\beta)$ and $(X',i';\beta')$ are isomorphic.
\end{definition}

\begin{definition}
Let $\mathcal{M}_{S_G}$ be the subspace of $\mathcal{M}$ corresponding to the marked pairs with the condition $R$ given by $\{\} \subset S_G \subset L$.  Let $\mathcal{N}_{S_G}$ be the image of $\mathcal{M}_{S_G}$ under the forgetful map.
\end{definition}

Given a marked pair with symplectic automorphism group $G$ and action $\phi$, we may obtain a marked pair with condition $R$.  The global Torelli theorem for K3 surfaces implies that two marked pairs with symplectic automorphism group $G$ and action $\phi$ correspond to the same marked pair with condition $R$ if and only if they are isomorphic.  Let $\mathcal{M}_{G,\phi}$ be the subspace of $\mathcal{M}_{S_G}$ corresponding to marked pairs with symplectic automorphism group $G$ and action $\phi$; by \cite[Theorem 4.10]{Nikulin}, $\mathcal{M}_{G,\phi}$ is open in $\mathcal{M}_{S_G}$.  Let $\mathcal{N}_{G,\phi}$ be the image of $\mathcal{M}_{G,\phi}$ under the forgetful map, and let $\mathcal{X}_{G,\phi}$ be the subset of the universal family $\mathcal{X}$ of marked K3 surfaces lying over $\mathcal{M}_{G,\phi}$.


Consider the minimal resolution $Y_\nu$ of the quotient $X_\nu/G$.  Let $\Delta = \{C_i\}$ be the exceptional curves of $Y_\nu$.

\begin{definition}\cite[\S 4]{ShiZhang}
We say that a simple normal crossing divisor $\Delta$ on a K3 surface Y is an \emph{ADE configuration
of smooth rational curves}, or, more briefly, an \emph{ADE configuration}, if each irreducible component of $\Delta$ is a smooth rational curve and the intersection matrix of the irreducible components of $\Delta$ is a direct sum of the Cartan matrices of type $A_l$, $D_m$ or $E_n$.  (We take these matrixes to be negative definite.)
\end{definition}

Fix a marking $\alpha_\nu$ of $Y_\nu$.  We may use the image $\alpha_\nu(\Delta)$ of the ADE configuration $\Delta$ to define a condition $T$; the marked K3 surface $(Y_\nu;\alpha_\nu)$ corresponds to a point in $\mathcal{N}_T$.  We wish to show that all points of $\mathcal{N}_T$ correspond to resolutions of symplectic quotients of K3 surfaces.  We will need the following classification of the covering spaces of the complements of ADE configurations of rational curves on K3 surfaces:

\begin{theorem}\label{T:shimadaZhangCampana}\cite[Proposition 4.1]{ShiZhang} \cite{Campana}
Let $\Delta$ be an ADE configuration of smooth rational curves on a K3 surface $Y$.  Let $Y' = Y - \Delta$, and let $X'$ be the universal covering space of $Y'$.  Then $X'$ and $\pi_1(X')$ satisfy one of the following conditions:

\begin{enumerate}
\item $X' \cong Y'$ and $\pi_1(Y')$ is trivial.
\item $X'$ is isomorphic to the complement of a discrete set of points $A$ in $\mathbf{C}^2$, and $\pi_1(Y')$ is infinite.  Furthermore, there exists a map $f$ from $\mathbf{C}^2 - A$ to a two-dimensional complex torus $T$ and a map $g$ from $T$ to $Y'$ such that $g$ is the quotient of $T$ by a finite group of automorphisms $\Gamma$ and $g \circ f$ is the covering map.
\item $X'$ is isomorphic to a K3 surface with a finite set of points removed, and the group of covering transformations (which is naturally isomorphic to $\pi_1(Y')$) acts symplectically on this surface.
\end{enumerate}
\end{theorem}

\begin{theorem}\label{T:allQuotients}
Suppose there exists $\nu \in \mathcal{N}_T$, corresponding to a marked K3 surface $(Y_{\nu}, \alpha_{\nu})$, such that $Y_{\nu}$ is the resolution of the quotient of a K3 surface $X_{\nu}$ by a symplectic $G$-action.  Let $q \in \mathcal{N}_T$, and let $(Y_q, \alpha_q)$ be the corresponding marked K3 surface.  Then there exists a K3 surface $X_q$ and a symplectic action of $G$ on $X_q$ such that $Y_q$ is a resolution of $X_q/G$.  
\end{theorem}

\begin{proof}
For any $n \in \mathcal{N}_T$, we may choose a neighborhood $U_n$ of $n$ such that for all $n'$ in $U_n$, there exists a diffeomorphism $\alpha: Y_n \to Y_{n'}$ such that $\alpha^*(\alpha_{n'}^{-1}(M)) = \alpha_{n}^{-1}(M)$ and (since rational curves in K3 surfaces are uniquely determined by their homology classes) $\alpha^*(\alpha_{n'}^{-1}(\{c_j\})) = \alpha_{n}^{-1}(\{c_j\})$.  Thus $Y_n - \alpha_m^{-1}(\{c_j\})$ and $Y_{n'} - \alpha_{n'}^{-1}(\{c_j\})$ are isomorphic, and $\pi_1(Y_n - \alpha_n^{-1}(\{c_j\})) = \pi_1(Y_{n'} - \alpha_{n'}^{-1}(\{c_j\}))$.

By Theorem~\ref{T:components}, there exists a path in $\mathcal{N}_T$ from $q$ to $\nu$.  Covering this path by a finite number of the neighborhoods $U_n$, we see that $\pi_1(Y_q - \alpha_q^{-1}(\{c_j\}))$ is isomorphic to $\pi_1(Y_\mu - \alpha_\mu^{-1}(\{c_j\}))$, so $\pi_1(Y_q - \alpha_q^{-1}(\{c_j\}))= G$.  By Theorem~\ref{T:shimadaZhangCampana}, the covering space of $Y_q - \alpha_q^{-1}(\{c_j\})$ is isomorphic to a K3 surface $X_q$ with a finite number of points removed, and $G$ acts symplectically on $X_q$.  Thus, $Y_q$ is the resolution of $X_q/G$, as desired.
\end{proof}

Starting with $Y_q$, we obtained a pair $(X_q, i_q : G \hookrightarrow \mathrm{Aut}\,X)$.  



\begin{definition}
We say that two points $n,n' \in \mathcal{N}_T$ \emph{determine the same action of} $G$ on the two-dimensional integral cohomology of K3 surfaces if there exist corresponding pairs $(X_n, i_n : G \hookrightarrow \mathrm{Aut}\,X)$, $(X_{n'}, i_{n'} : G \hookrightarrow \mathrm{Aut}\,X)$ and an isomorphism $\phi: H^2(X_n, \mathbf{Z}) \to H^2(X_{n'},\mathbf{Z})$ which preserves the cup product and satisfies the relation
\[i_{n'}(g)^* = \phi \cdot i_n(g) \cdot \phi^{-1}\]
for any $g \in G$.
\end{definition}

The condition that points determine the same action of $G$ defines an equivalence relation on $\mathcal{N}_T$.




\begin{theorem}\label{T:open}
Let $n_0 \in \mathcal{N}_T$, and suppose $n_0$ corresponds to the pair $(X_n, i_n : G \hookrightarrow \mathrm{Aut}\,G)$.  The set of points in $\mathcal{N}_T$ which determine the same action of $G$ as $n_0$ is open.
\end{theorem}

\begin{proof}
We construct an open neighborhood of $n_0$ in which the action coincides with the action determined by $n_0$ and a corresponding neighborhood in $\mathcal{N}_{G,\phi}$.  Fix a marking $\beta_{n_0}: H^2(X,\mathbf{Z}) \to L$.  The triple $(X_{n_0}, i_{n_0}, \beta_{n_0})$ defines a point $\nu_0$ in the moduli space $\mathcal{N}_{G,\phi}$.  The usual map $u: \mathcal{X} \to \mathcal{N}$ restricts to a map $u_{G,\phi}: \mathcal{X}_{G,\phi} \to \mathcal{N}_{G,\phi}$.  Following \cite[\S 8.5]{Nikulin}, we obtain a neighborhood $V$ of $\nu_0$ in $\mathcal{N}_{G,\phi}$, a corresponding neighborhood $\mathcal{X}_{G,\phi}^V$ in $\mathcal{X}_{G,\phi}$, and a resolution $\mathcal{Y}_{G,\phi}^V$ of $\mathcal{X}_{G,\phi}^V/G$ such that the following diagram commutes:

\[\begindc{0}[5]
\obj(0,10)[]{$\mathcal{X}_{G,\phi}^V$}
\obj(20,10)[]{$\mathcal{X}_{G,\phi}^V/G$}
\obj(40,10)[]{$\mathcal{Y}_{G,\phi}^V$}
\obj(20,0)[]{$V$}
\mor(1,10)(18,10){$\pi$}
\mor(40,10)(22,10){$\sigma$}[\atright,\solidarrow]
\mor(0,10)(20,0){$u_{G,\phi}$}[\atright,\solidarrow]
\mor(20,10)(20,0){}
\mor(40,10)(20,0){$v$}
\enddc\]

Each curve $E_j$ in $Y_{n_0}$ extends uniquely to an effective divisor $\mathbf{E}_j$ on $\mathcal{Y}_{G,\phi}^V$.  For each $\nu \in V$, $\mathbf{E}_j \cdot Y_\nu = E_j^\nu$ is a nonsingular rational curve on $Y_n$, where $\{E_j^\nu\}$ is the set of components of the curves obtained from the resolution of singularities of $X_\nu / G$.  We set $\mathcal{X}_{G,\phi}' = \mathcal{X}_{G,\phi} - \{\text{fixed~points~of}~G\}$ and $\mathcal{Y}_{G,\phi}' = \mathcal{Y}_{G,\phi} - \cup \mathbf{E}_j$, obtaining a new commutative diagram:

\[\begindc{0}[5]
\obj(0,10)[]{$\mathcal{X}_{G,\phi}^V$}
\obj(12,10)[]{$(\mathcal{X}_{G,\phi}^V)'$}
\obj(28,10)[]{$(\mathcal{Y}_{G,\phi}^V)'$}
\obj(40,10)[]{$\mathcal{Y}_{G,\phi}^V$}
\obj(20,0)[]{$V$}
\mor(1,10)(10,10){}
\mor(40,10)(30,10){}
\mor(0,10)(20,0){$u_{G,\phi}$}[\atright,\solidarrow]
\mor(15,10)(25,10){$\pi'$}
\mor(10,10)(20,0){$u_{G,\phi}'$}
\mor(30,10)(20,0){$v'$}[\atright,\solidarrow]
\mor(40,10)(20,0){$v$}
\enddc\]

These maps induce corresponding maps on $G$-sheaves:

\[{R^2u_{G,\phi}}_*\mathbf{Z} \overset{i^*}{\longrightarrow} {R^2u_{G,\phi}}_*'\mathbf{Z} \overset{\pi'^*}{\longleftarrow} {R^2v_{G,\phi}}_*'\mathbf{Z} \overset{j^*}{\longleftarrow} {R^2v_{G,\phi}}_*\mathbf{Z} \]

\cite{Nikulin} showed that there exists a map 
\[\theta = (i^*)^{-1} \circ \pi'^* \circ \bar{j^*}: {R^2v_{G,\phi}}_*\mathbf{Z}/\oplus \mathbf{ZE}_j \to ({R^2v_{G,\phi}}_*\mathbf{Z})^G\] 
which satisfies $\theta(x) \cdot \theta(y) = |G| (x\cdot y)$ for $x,y \in (\oplus\mathbf{ZE}_j)^\perp$ and fits into an exact sequence
\[
0 \to \mathrm{ker}~\theta  \longrightarrow R^2v_*\mathbf{Z}/\oplus \mathbf{Z}E_j \overset{\theta}{\longrightarrow} (R^2u_{G,\phi_*}\mathbf{Z})^G.
\]
\cite{Nikulin} also showed that $\mathrm{ker}~\theta$ is the torsion subsheaf of $R^2v_*\mathbf{Z}/\oplus \mathbf{Z}E_j$.

Over $\mu_0$, we may use the markings $\alpha_{n_0}$ and $\beta_{n_0}$ to obtain the exact sequence
\[\begin{CD}
0 @>>> M/\oplus\mathbf{Z}c_j @>>> L/\oplus\mathbf{Z}c_j @>\beta_{n_0} \circ \theta \circ \alpha_{n_0}^{-1}>> L^{\phi(G)}.
\end{CD}\]

Note that $\theta$ restricts to an injective map $\theta: M^\perp \hookrightarrow L^{\phi(G)}$.  Proposition~\ref{P:NikWhit} implies that $M^\perp$ and $L^{\phi(G)}$ have the same rank, so we may extend $\theta$ to an isomorphism from $M^\perp \otimes \mathbf{C}$ to $L^{\phi(G)} \otimes \mathbf{C}$.  The theorem now follows from the argument in the abelian case (see \cite[\S 8.5]{Nikulin}).
\end{proof}

\begin{remark}
\cite{Nikulin} claimed that $\theta$ is a surjective map when $G$ is abelian.  As Garbagnati and Sarti observed, this is not the case (see \cite[Proposition 2.4]{Garbagnati} and \cite{GarSarti}).  In general, the discrepancy is given by Theorem~\ref{T:exactSequence}. 
\end{remark}

\begin{corollary}\label{C:sameAction}
All points of $\mathcal{N}_T$ determine the same action.
\end{corollary}

Together, Theorem~\ref{T:allQuotients} and Corollary~\ref{C:sameAction} show that we may classify symplectic actions on K3 surfaces by classifying the conditions $T$ which are obtained from symplectic actions.  \cite[Table 2]{Xiao} lists the ADE configurations corresponding to finite groups which can act symplectically; we shall refer to these ADE configurations as \emph{symplectic ADE configurations}.  In most cases, a group $G$ corresponds to a single symplectic ADE configuration; the exceptions are $Q_8$, the group of unit quaternions, and $T_{24}$, the binary tetrahedral group of order $24$, each of which corresponds to two different symplectic ADE configurations.  Nikulin showed by direct computation that when $G$ is abelian, the primitive lattice $M$ generated by the singular curves has a unique embedding in the K3 lattice, so $T$ is uniquely determined by $G$ (see \cite[Theorem 7.2]{Nikulin}).  The condition $T$ (and thus the action of $G$) is not uniquely determined by $G$ for every non-abelian group $G$.  For instance, Hashimoto showed in \cite[Proposition~2.12]{Hashimoto} that there are two distinct symplectic actions of the symmetric group $G=\mathcal{S}_5$.

\begin{question}\label{Q:embedding}
Does every embedding of a symplectic ADE configuration of rational curves in the K3 lattice yield a symplectic group action?  
\end{question}

Theorem~\ref{T:shimadaZhangCampana} tells us that we may approach Question~\ref{Q:embedding} by analyzing the possible fundamental groups of the complement of a given configuration.  Let $T^2$ be a two-dimensional complex torus, and let $\Gamma$ be a finite group of automorphisms of $T^2$.  Fujiki classified the possible finite groups $\Gamma$ in \cite{Fujiki}, and Bertin, \"{O}nsiper, and Sert{\"{o}}z classified the resulting singularities of $T^2/\Gamma$ (see \cite{Bertin} and \cite[Proposition 3]{OnSert}):

~\\
\begin{tabular}{l l}
Group & Singularities of $T^2/\Gamma$ \\ \hline
$C_2$ & $16 A_1$ \\
$C_3$ & $9 A_2$  \\
$C_4$ & $4 A_3 + 6 A_1$ \\
$C_6$ & $A_5 + 4 A_2 + 5 A_1$ \\
$Q_8$ & $4D_4 + 3A_1$ \\
$Q_{12}$ & $D_5 + 3A_3 + 2A_2 + A_1$ \\
$T_{24}$ & $A_5 + 2 A_3 + 4 A_2$ \\
$T_{24}$ & $E_6 + D_4 + 4A_2 + A_1$
\end{tabular}
~\\

\noindent (Here $C_k$ is the cyclic group of order $k$, $Q_8$ and $Q_{12}$ are binary dihedral groups, and $T_{24}$ is the binary tetrahedral group.)

The list of K3 singularities obtained from group actions in \cite[Table 2]{Xiao} is disjoint from the list above. 

We next consider whether there exists an ADE configuration $\Delta$ which can be obtained in two ways: from a singular K3 surface whose smooth part has trivial fundamental group, and as the ADE singularity of another K3 surface whose smooth part has non-trivial fundamental group.  Most of the cases can be eliminated using the following lemma, as stated in \cite[Lemma 4.6]{ShiZhang}:

\begin{lemma}\cite[Lemma 2]{Xiao}\label{L:fundGroupTrivial}
Let $\tilde{\Delta}$ be an ADE configuration of rational curves on a K3 surface, let $\mathbf{Z}[\Delta]$ be the sublattice of the K3 lattice $L$ generated by the curves in $\tilde{\Delta}$, and let $M_\Delta$ be the smallest primitive sublattice of $L$ containing $\mathbf{Z}[\Delta]$.  Then the dual of the abelianisation of $\pi_1(X - \tilde{\Delta})$ is canonically isomorphic to $M_\Delta / \mathbf{Z}[\Delta]$.  In particular, if $\pi_1(X - \tilde{\Delta})$ is trivial, then $\mathbf{Z}[\Delta]$ embeds in $L$ as a primitive sublattice.
\end{lemma}

\cite[Table 2]{Xiao} lists $M_\Delta / \mathbf{Z}[\Delta]$ for each ADE configuration which can occur as the exceptional divisor of a resolution of the quotient of a K3 surface by a group of symplectic automorphisms.  Using Lemma~\ref{L:fundGroupTrivial}, we conclude that none of the configurations in \cite[Table 2]{Xiao} can yield a trivial fundamental group, save possibly the following list of symplectic ADE configurations obtained from perfect groups:

~\\
\begin{tabular}{l l}
Group & Symplectic ADE Configuration\\ \hline
$\mathcal{A}_5$ & $2A_4 + 3A_2 + 4A_1$\\
$L_2(7)$ & $A_6 + 2A_3 + 3A_2 + A_1$\\
$\mathcal{A}_6$ & $2A_4 + 2 A_3 + 2 A_2 + A_1$\\
$M_{20}$ & $D_4 + 2A_4 + 3 A_2 + A_1$
\end{tabular}
~\\

\noindent(Here $\mathcal{A}_5$ and $\mathcal{A}_6$ are alternating groups, $L_2(7)$ is the Chevalley group $PSL(2,\mathbf{F}_7)$, and $M_{20}$ is a subgroup of the Mathieu group $M_{24}$ which is isomorphic to the semidirect product $(\mathbf{Z}/2\mathbf{Z})^4 \ltimes \mathcal{A}_5$.)

Symplectic actions of these groups have been extensively studied using Niemeier lattices.  Mukai studied the lattice invariants of $S_G$ when $G = M_{20}$ in an appendix to \cite{Kondo}; Oguiso and Zhang investigated finite non-symplectic extensions of an $L_2(7)$ action in \cite{OZ}; Keum, Oguiso, and Zhang studied extensions of $\mathcal{A}_6$ actions in \cite{KOZ} and \cite{KOZ2}; and Hashimoto considered actions induced by $\mathcal{A}_5 \hookrightarrow \mathcal{S}_5$ in \cite{Hashimoto}.

When $G = \mathcal{A}_5$, the lattice $M = M_\Delta$ has rank $18$ and discriminant group $M^*/M \cong (\mathbf{Z}/5\mathbf{Z})^2 \oplus (\mathbf{Z}/3\mathbf{Z})^3 \oplus (\mathbf{Z}/2\mathbf{Z})^4$.  Therefore, the primitive embedding of $M$ in the K3 lattice $L$ is unique up to isometries of $L$ by the results of \cite{NikForms}, and $\mathcal{A}_5$ corresponds to a single condition $T$ and moduli space $\mathcal{M}_T$.

For each of the groups $\mathcal{A}_6$, $L_2(7)$, and $M_{20}$, the lattice $M$ has rank $19$; thus, its orthogonal complement $M^\perp$ in $L$ will be a positive definite lattice of rank $3$.  Isomorphism classes of positive definite lattices are not always uniquely determined by their invariants.  Using the computer algebra system Magma, one may check that when $M_{20}$ acts symplectically, the lattice $M^\perp$ is uniquely determined up to isomorphism (see \cite{Magma}).  However, in the cases of $\mathcal{A}_6$ and $L_2(7)$ a similar analysis in Magma shows that there are two distinct candidates for each $M^\perp$, and therefore two possible embeddings of each lattice $M$ in $L$ (up to overall isometry).  Determining whether these embeddings can be constructed using symplectic actions of $\mathcal{A}_6$ and $L_2(7)$ is an interesting question for further research.

\end{document}